\documentclass[12pt,twoside]{amsart}
\usepackage{geometry}
\usepackage{amssymb,amsmath,amsthm, amscd, enumerate, mathrsfs}
\usepackage{graphicx, hhline}
\usepackage[all]{xy}
\usepackage[dvipdfmx]{hyperref}

\title[Vanishing theorems for projective 
morphisms]{Vanishing theorems for projective morphisms between 
complex analytic spaces}
\author{Osamu Fujino}
\date{2023/10/14, version 0.15}
\address{Department of 
Mathematics, Graduate School of Science, 
Kyoto University, Kyoto 606-8502, Japan}
\email{fujino@math.kyoto-u.ac.jp}

\DeclareMathOperator{\codim}{codim}
\DeclareMathOperator{\Supp}{Supp}
\DeclareMathOperator{\Specan}{Specan}
\DeclareMathOperator{\Ass}{Ass}
\DeclareMathOperator{\Pic}{Pic}
\DeclareMathOperator{\mult}{mult}
\newtheorem{thm}{Theorem}[section]
\newtheorem{lem}[thm]{Lemma}

\newtheorem{prop}[thm]{Proposition}

\theoremstyle{definition}
\newtheorem{defn}[thm]{Definition}
\newtheorem{rem}[thm]{Remark}
\newtheorem*{ack}{Acknowledgments}  
\newtheorem{step}{Step}
\newtheorem{say}[thm]{}
\makeatletter
    
    \@addtoreset{equation}{section}
\makeatother
\begin{document}

\maketitle 

\begin{abstract} 
We discuss vanishing theorems for projective morphisms between 
complex analytics spaces and some related results. 
They will play a crucial role in the minimal 
model theory for projective morphisms of complex analytic spaces. 
Roughly speaking, we establish an ultimate generalization of 
Koll\'ar's package from the minimal model theoretic viewpoint. 
\end{abstract}

\tableofcontents

\section{Introduction}\label{a-sec1}
In \cite{fujino-minimal}, we have already discussed the minimal 
model program for kawamata log terminal pairs in a complex 
analytic setting. Roughly speaking, we showed that 
\cite{bchm} and \cite{hacon-mckernan} can work for 
projective morphisms between complex analytic spaces. 
We note that one of the main ingredients of \cite{bchm} and 
\cite{hacon-mckernan} is the Kawamata--Viehweg vanishing theorem and that 
the Kawamata--Viehweg vanishing theorem can be formulated 
and proved for projective morphisms of complex analytic spaces. 
Hence the results in \cite{fujino-minimal} are not 
surprising although they are nontrivial. 
In \cite{fujino-fundamental} and \cite[Chapter 6]{fujino-foundations}, 
we established the fundamental theorems of the minimal model 
program for log canonical pairs and quasi-log schemes, 
respectively. It is highly desirable to generalize \cite{fujino-fundamental} and 
\cite[Chapter 6]{fujino-foundations} into the complex analytic setting. 
For that purpose, we have to establish suitable vanishing theorems 
and some related results for projective morphisms of complex 
analytic spaces. 
Therefore, in this paper, we prove the following theorems 
(see Theorems \ref{a-thm1.1} and \ref{a-thm1.2}), 
which give a complete answer to \cite[Remark 5.8.3]{fujino-foundations}. 
They will play a crucial role for the study of complex analytic log canonical 
pairs and quasi-log structures on complex analytic spaces 
(see \cite{fujino-fundamental} and \cite[Chapter 6]{fujino-foundations}). 
In \cite{fujino-cone-contraction}, we will establish the 
cone and contraction theorem of normal pairs for 
projective morphisms between complex analytic spaces as an application 
of this paper. 
In \cite{fujino-quasi-log-analytic}, we will discuss quasi-log structures 
for complex analytic spaces by using 
Theorems \ref{a-thm1.1} and \ref{a-thm1.2}. 

\begin{thm}[Main theorem]\label{a-thm1.1}
Let $(X, \Delta)$ be an analytic simple 
normal crossing pair such that $\Delta$ is a boundary 
$\mathbb R$-divisor on $X$. 
Let $f\colon X\to Y$ be a projective morphism 
to a complex analytic space $Y$ and let $\mathcal L$ 
be a line bundle on $X$. 
Let $q$ be an arbitrary nonnegative integer. 
Then we have the following properties. 
\begin{itemize}
\item[(i)] $($Strict support condition$)$.~If 
$\mathcal L-(\omega_X+\Delta)$ is $f$-semiample,  
then every 
associated subvariety of $R^qf_*\mathcal L$ is the $f$-image 
of some stratum of $(X, \Delta)$. 
\item[(ii)] $($Vanishing theorem$)$.~If 
$\mathcal L-(\omega_X+\Delta)\sim _{\mathbb R} f^*\mathcal H
$ holds 
for some $\pi$-ample 
$\mathbb R$-line bundle $\mathcal H$ on $Y$, where 
$\pi\colon Y\to Z$ is a 
projective morphism to a complex analytic space 
$Z$, then we have 
$
R^p\pi_*R^qf_*\mathcal L=0
$
for every $p>0$. 
\end{itemize} 
\end{thm}

\begin{thm}[Vanishing theorem of Reid--Fukuda type]\label{a-thm1.2}
Let $(X, \Delta)$ be an analytic simple 
normal crossing pair such that $\Delta$ is a boundary 
$\mathbb R$-divisor on $X$. 
Let $f\colon X\to Y$ and $\pi\colon Y\to Z$ be projective morphisms 
between complex analytic spaces and let $\mathcal L$ 
be a line bundle on $X$. 
If $\mathcal L-(\omega_X+\Delta)\sim _{\mathbb R} f^*\mathcal H$ 
holds such that $\mathcal H$ is an $\mathbb R$-line bundle, 
which is nef and 
log big over $Z$ with respect to $f\colon (X, \Delta)\to Y$, on $Y$, then  
$R^p\pi_*R^qf_*\mathcal L=0$ holds for every $p>0$ and every $q$. 
\end{thm}

We make an important remark on Theorems \ref{a-thm1.1} 
and \ref{a-thm1.2}. 

\begin{rem}\label{a-rem1.3}
(i) Let $X$ be a complex analytic space and let $\Pic(X)$ 
be the group of line bundles on $X$, 
that is, the {\em{Picard group}} of $X$. An element 
of $\Pic(X)\otimes _{\mathbb Z}\mathbb R$ 
(resp.~$\Pic(X)\otimes _{\mathbb Z}\mathbb Q$) is called 
an {\em{$\mathbb R$-line bundle}} 
(resp.~a {\em{$\mathbb Q$-line bundle}}) on $X$. 
In this paper, we write the group law of 
$\Pic(X)\otimes _{\mathbb Z}\mathbb R$ additively 
for simplicity of notation. 

(ii) In Theorems \ref{a-thm1.1} and \ref{a-thm1.2}, we 
always assume that $\Delta$ is {\em{globally $\mathbb R$-Cartier}}, 
that is, $\Delta$ is a finite $\mathbb R$-linear combination 
of Cartier divisors. We note that 
if the number of the irreducible components of $\Supp 
\Delta$ is finite then $\Delta$ is globally $\mathbb R$-Cartier. 
This condition is harmless to 
applications 
because the restriction of $\Delta$ to a relatively compact 
open subset of $X$ has only 
finitely many irreducible components in its support. 

(iii) Under the assumption that 
$\Delta$ is globally $\mathbb R$-Cartier, 
we can obtain an $\mathbb R$-line bundle 
$\mathcal N$ naturally associated to $\mathcal L-(\omega_X+\Delta)$, 
which is a hybrid of line bundles $\mathcal L$ and 
$\omega_X$ and a globally $\mathbb R$-Cartier divisor 
$\Delta$. 
The assumption in Theorem \ref{a-thm1.1} (i) means that 
$\mathcal N$ is a finite positive $\mathbb R$-linear 
combination of $f$-semiample line bundles on $X$. 
The assumption in Theorem \ref{a-thm1.1} (ii) and Theorem \ref{a-thm1.2} 
means that $\mathcal N=f^*\mathcal H$ holds 
in $\Pic(X)\otimes _{\mathbb Z} \mathbb R$.  
\end{rem}

We will use Theorems \ref{a-thm1.1} and \ref{a-thm1.2} 
in order to translate various results on log canonical pairs 
in \cite{fujino-fundamental} 
and quasi-log schemes in \cite[Chapter 6]{fujino-foundations} 
into the complex analytic setting (see \cite{fujino-cone-contraction} 
and \cite{fujino-quasi-log-analytic}). 
The proof of Theorems \ref{a-thm1.1} and \ref{a-thm1.2} 
is different from the corresponding one in 
\cite{fujino-fundamental} and \cite[Chapter 5]{fujino-foundations} 
in the algebraic setting. 
In this paper, we use a spectral sequence coming from 
Saito's theory of mixed Hodge modules 
(see \cite{saito1}, \cite{saito2}, \cite{saito-tohoku}, 
\cite{fujino-fujisawa-saito}, and \cite{saito5}). 
Roughly speaking, we reduce the problem 
to a well-known simpler case by semisimplicity 
of polarizable Hodge modules. 
Then we use Takegoshi's complex analytic generalization of Koll\'ar's 
package (see \cite{takegoshi}). 
The approach in this paper clarifies the meaning 
of the strict support condition in Theorem \ref{a-thm1.1} (i). 
We strongly recommend the interested reader to compare 
this paper with \cite[Chapter 5]{fujino-foundations}. 
We note that the reader can find an alternative 
approach to Theorems \ref{a-thm1.1} and \ref{a-thm1.2}, 
which is free from Saito's theory of mixed Hodge modules, 
in \cite{fujino-fujisawa}. 

\medskip 

Let us quickly explain various related vanishing theorems. 
We recommend the interested reader to 
see \cite[Chapter 3]{fujino-foundations}. 
There are many results on vanishing theorems. 
Our choice of topics here is biased and reflects author's personal 
taste. The author apologizes for many important 
omitted references. 

\begin{say}[Kodaira vanishing theorem]\label{a-say1.4}
The Kodaira vanishing theorem (see 
\cite{kodaira}) is a monumental 
result and is very important 
in the study of complex algebraic varieties. 
We have many useful and powerful generalizations. 
\end{say}

\begin{say}[Kawamata--Viehweg vanishing theorem]\label{a-say1.5}
The Kawamata--Viehweg vanishing theorem (see \cite{kawamata} 
and \cite{viehweg}) is one of the most famous 
generalizations of the Kodaira vanishing theorem. 
It plays a crucial role in the minimal model 
theory for higher-dimensional complex algebraic varieties with 
only mild singularities.  
\end{say}

\begin{say}[Koll\'ar's injectivity, torsion-free, 
and vanishing theorem]\label{a-say1.6}
In \cite{kollar1}, J\'anos Koll\'ar obtained a very powerful generalization of 
the Kodaira vanishing theorem. 
The reader can find simpler approaches and 
some generalizations in \cite{esnault-viehweg} and \cite{kollar3}. 
The original approach in \cite{kollar1} depends on the 
theory of pure Hodge structures. 
By using the theory of mixed Hodge structures, we 
can prove some generalizations (see \cite{esnault-viehweg}, 
\cite{fujino-higher}, \cite{ambro}, \cite[Chapter 5]{fujino-foundations}, 
and so on). 
They have already had many applications 
in the study of minimal models of complex algebraic 
varieties with bad singularities. 
For the details, see \cite{fujino-higher}, \cite{ambro}, 
\cite{fujino-on-injectivity}, \cite{fujino-fundamental}, 
\cite{fujino-fundamental-slc}, \cite{fujino-kodaira}, \cite{fujino-vanishing}, 
\cite[Chapters 5 and 6]{fujino-foundations}, \cite{fujino-on-semipositivity}, 
\cite{fujino-injectivity}, \cite{fujino-vani-semi}, and so on. 
\end{say}

\begin{say}[Koll\'ar's conjecture]\label{a-say1.7}
In \cite[Section 5]{kollar2}, 
Koll\'ar discussed some conjecture about 
abstract variations of Hodge structure, which is now 
usually called {\em{Koll\'ar's conjecture}}. 
In \cite{saito3}, Morihiko Saito showed that it naturally follows from 
the general theory of Hodge modules (see \cite{saito1} and 
\cite{saito2}). Note that \cite[Section 4]{fujino-higher}, 
which was written by Saito, and \cite{fujino-fujisawa-saito} 
are closely related to Koll\'ar's conjecture. 
We also note that the main result in this paper heavily depends 
on \cite{fujino-fujisawa-saito} (see also \cite{saito5}). 
\end{say}

\begin{say}[Kodaira--Saito vanishing theorem]\label{a-say1.8}
Saito established a powerful generalization of the Kodaira vanishing 
theorem in the framework of mixed Hodge modules 
(see \cite[(2.g) Kodaira vanishing]{saito2}). 
For the details and some further generalizations, see 
\cite{popa}, \cite{schnell-saito}, \cite{suh}, \cite{wu}, and so on. 
Although we use the result in \cite{fujino-fujisawa-saito} for 
the proof of the main theorem of this paper, we do not 
directly use the Kodaira--Saito vanishing theorem. 
\end{say}

\begin{say}[Takegoshi's analytic generalization 
of Koll\'ar's theorem]\label{a-say1.9}
In \cite{takegoshi}, 
Kensho Takegoshi established Koll\'ar's injectivity, torsion-free, and 
vanishing theorem in a suitable complex analytic setting. 
We use Takegoshi's result in this paper.  
For various generalizations of Koll\'ar's package in the 
complex analytic setting, see \cite{fujino-trans1}, 
\cite{fujino-trans2}, \cite{fujino-matsumura}, 
\cite{fujino-asia}, \cite{matsumura}, \cite{fujino-relative}, and so on. 
\end{say}

We look at the organization of this paper. 
In Section \ref{a-sec2}, we collect some basic definitions and 
necessary results. 
In Section \ref{a-sec3}, we treat a very standard setting of 
proper K\"ahler morphisms. 
We establish the strict support condition, 
the vanishing theorem, and the injectivity theorem 
in this standard setting. 
In Section \ref{a-sec4}, we prove Theorem \ref{a-thm1.1} 
under the extra assumption that 
$X$ is irreducible. We note that the proof of Theorem \ref{a-thm1.1} 
for the case where $X$ is irreducible is essential. 
In Section \ref{a-sec5}, we prove 
Theorem \ref{a-thm1.1} in full generality. 
More precisely, we prepare a technical but important lemma 
and show that essentially the same argument as in Section \ref{a-sec4} 
works. 
In Section \ref{a-sec6}, we prove Theorem \ref{a-thm1.2}. 
Although it looks very similar to 
Theorem \ref{a-thm1.1} (ii), the proof of Theorem \ref{a-thm1.2} 
is much harder than that of Theorem \ref{a-thm1.1} (ii). 

\begin{ack}\label{a-ack}
The author was partially 
supported by JSPS KAKENHI Grant Numbers 
JP19H01787, JP20H00111, JP21H00974, JP21H04994. 
He would like to thank Professor Morihiko Saito very much 
for useful 
discussions, suggestions, and answering many questions. 
He would also like to thank Professors Taro Fujisawa, Shihoko Ishii, 
Shin-ichi Matsumura, Takeo Ohsawa, and Shigeharu Takayama. 
\end{ack}

In this paper, every complex analytic space is assumed to be 
{\em{Hausdorff}} and {\em{second-countable}}. 
We will freely use the standard notation in \cite{fujino-fundamental}, 
\cite{fujino-foundations}, \cite{fujino-minimal}, and so on. 
We will also freely use the basic results on complex analytic geometry 
in \cite{banica} and \cite{fischer}. 
For the minimal model program for projective 
morphisms between complex analytic spaces, 
see \cite{nakayama1}, \cite{nakayama2}, and \cite{fujino-minimal}. 

\section{Preliminaries}\label{a-sec2}
In this section, we will collect some basic definitions and 
necessary results. Let us start with the definition of 
{\em{analytic simple normal crossing pairs}}. 

\begin{defn}[Analytic simple normal crossing pairs]\label{a-def2.1}
Let $X$ be a simple normal crossing divisor 
on a smooth complex analytic space $M$ and 
let $B$ be an $\mathbb R$-divisor on $M$ such that 
$\Supp (B+X)$ is a simple normal crossing divisor on $M$ and 
that $B$ and $X$ have no common irreducible components. 
Then we put $D:=B|_X$ and 
consider the pair $(X, D)$. 
We call $(X, D)$ an {\em{analytic globally embedded simple 
normal crossing pair}} and $M$ the {\em{ambient space}} 
of $(X, D)$. 

If the pair $(X, D)$ is locally isomorphic to an analytic 
globally embedded 
simple normal crossing pair and the irreducible 
components of $X$ and $D$ are all smooth, 
then $(X, D)$ is called an {\em{analytic simple normal crossing 
pair}}. If $(X, 0)$ is an analytic simple normal crossing pair, 
then we simply say that $X$ is {\em{simple normal crossing}}. 

When $X$ is simple normal crossing, it has an invertible 
dualizing sheaf $\omega_X$. 
We sometimes use the symbol $K_X$ as a formal 
divisor class with an isomorphism 
$\mathcal O_X(K_X)\simeq \omega_X$ if there is 
no danger of confusion. 
We note that we can not always define $K_X$ globally 
with $\mathcal O_X(K_X)\simeq \omega_X$. 
In general, it only exists locally on $X$. 
\end{defn}

We need the following definition in order to state 
\cite[Corollary 1 and 4.7.~Remark]{fujino-fujisawa-saito}. 

\begin{defn}\label{a-def2.2}
Let $(X, D)$ be an analytic simple normal crossing pair such that 
$D$ is reduced. 
For any positive integer $k$, 
we put 
\begin{equation*}
X^{[k]}:=\left\{x\in X |\mult_x X\geq k\right\}^{\sim},  
\end{equation*}
where $Z^{\sim}$ denotes the normalization of $Z$. 
Then $X^{[k]}$ is the disjoint union of the intersections of $k$ irreducible 
components of $X$, and is smooth. 
We have a reduced simple normal crossing divisor 
$D^{[k]}\subset X^{[k]}$ defined by the pull-back of $D$ by the natural 
morphism 
$X^{[k]}\to X$. 
For any nonnegative integer $l$, 
we put 
\begin{equation*}
D^{[k, l]}:=\left\{\left. x\in X^{[k]} \right| \mult _x D^{[k]}\geq l\right\}^{\sim}. 
\end{equation*} 
We note that $D^{[k,0]}=X^{[k]}$ holds by definition. 
We also note that $\dim D^{[k, l]}=n+1-k-l$, where 
$n=\dim X$. 
\end{defn}

We recall the notion of {\em{strata}} of analytic simple normal crossing pairs. 

\begin{defn}[Strata]\label{a-def2.3}
Let $(X, D)$ be an analytic simple normal crossing pair. 
Let $\nu\colon X^\nu\to X$ be the normalization. 
We put $K_{X^\nu}+\Theta=\nu^*(K_X+D)$. 
This means that $\Theta$ is the union of $\nu^{-1}_*D$ and the 
inverse image of the singular locus of $X$. 
If $W$ is an irreducible component of $X$ 
or the $\nu$-image 
of some log canonical center of $(X^\nu, \Theta)$, 
then $W$ is called a {\em{stratum}} of $(X, D)$. 

When $D$ is reduced, $W$ is a stratum of $(X, D)$ if and only if 
$W$ is the image of an irreducible component of 
$D^{[k, l]}$ for some $k>0$ and $l\geq 0$. 
\end{defn}

The following easy remark may help the reader understand 
the notion of strata. 

\begin{rem}\label{a-rem2.4} 
Let $(X, D)$ be an analytic simple normal crossing 
pair. Let $D=\sum _i a_iD_i$ be the irreducible decomposition. 
We put $G:=\sum _{a_i=1}D_i$. 
Then $(X, G)$ is an analytic simple normal crossing pair 
such that $G$ is reduced. 
We can easily check that $W$ is a stratum of $(X, D)$ if and 
only if $W$ is a stratum of $(X, G)$. 
Therefore, a stratum $W$ of $(X, D)$ is the image 
of an irreducible component of $G^{[k, l]}$ for some $k>0$ and 
$l\geq 0$. 
\end{rem}

For Theorem \ref{a-thm1.2}, the notion of {\em{nef and log big 
$\mathbb R$-line bundles}} is necessary. 

\begin{defn}[Nef and log bigness]\label{a-def2.5}
Let $f\colon (X, \Delta)\to Y$ be a projective morphism 
from an analytic simple normal crossing pair $(X, \Delta)$ to 
a complex analytic space $Y$ and let $\pi\colon Y\to Z$ be 
a projective morphism between complex analytic spaces. 
Let $\mathcal H$ be an $\mathbb R$-line bundle on $Y$. 
We say that $\mathcal H$ is {\em{nef and 
log big over $Z$ with 
respect to $f\colon (X, \Delta)\to Y$}} if 
$\mathcal H|_{f(W)}$ is nef and big over $\pi\circ f(W)$ 
for every stratum $W$ of $(X, \Delta)$, 
equivalently, $\mathcal H\cdot C\geq 0$ holds for 
every projective integral curve $C$ on $Y$ such that 
$\pi(C)$ is a point and 
$\mathcal H|_{f(W)}$ can be written as a finite 
positive $\mathbb R$-linear combination of $\pi$-big line bundles on 
$f(W)$ for every stratum $W$ of $(X, \Delta)$. 
When $f$ is the identity morphism, 
we simply say that $\mathcal H$ is 
{\em{nef and log big over $Z$ with respect to $(X, \Delta)$}}. 
\end{defn}

One of the main ingredients of the proof of Theorem \ref{a-thm1.1} is 
a deep result coming from Saito's theory of 
mixed Hodge modules (see \cite{saito1}, \cite{saito2}, \cite{schnell-overview}, 
and so on). 
Roughly speaking, we reduce the problem to the 
case where $X$ is smooth by using Theorem \ref{a-thm2.6}. 
Then we use Takegoshi's results (see \cite{takegoshi} and 
Theorem \ref{a-thm2.9} below) in order 
to obtain Theorem \ref{a-thm1.1}. 

\begin{thm}
[{\cite[Corollary 1 and 4.7.~Remark]{fujino-fujisawa-saito} 
and 
\cite{saito5}}]\label{a-thm2.6}
Let $(X, D)$ be an analytic simple normal crossing pair 
with $\dim X=n$ such that 
$D$ is reduced and let $f\colon X\to Y$ 
be a proper morphism to a smooth 
complex analytic space $Y$. 
Assume that $f$ is K\"ahler on each irreducible 
component of $X$. 
Then there is the weight spectral sequence 
\begin{equation*}
{}_FE^{-q, i+q}_1=\bigoplus _{k+l=n+q+1}R^if_*\omega_{D^{[k, l]}/Y} 
\Rightarrow R^if_*\omega_{X/Y}(D), 
\end{equation*} 
degenerating at $E_2$, and its $E_1$-differential $d_1$ splits so that the 
${}_FE^{-q, i+q}_2$ are direct factors of ${}_FE^{-q, i+q}_1$. 
\end{thm}

We note that in \cite{saito5} a proper morphism $f\colon X\to Y$ of smooth 
complex analytic spaces 
is said to be {\em{K\"ahler}} if there exists a relative K\"ahler 
form $\xi_f$, which is a closed real $2$-form on $X$, satisfying 
the following condition that 
locally on $Y$ there is a K\"ahler form $\xi_Y$ such that $\xi_f+f^*\xi_Y$ 
is a K\"ahler form on $X$. 

\begin{rem}\label{a-rem2.7} 
Although we assume that $f\colon X\to Y$ is a projective morphism 
of complex algebraic varieties such that $Y$ is smooth 
in \cite[Corollary 1]{fujino-fujisawa-saito}, 
the results in \cite{fujino-fujisawa-saito} are also valid 
in the analytic case where $f$ 
is a proper K\"ahler morphism on each irreducible component of $X$, 
$Y$ is a complex manifold, and $(X, D)$ is an analytic 
simple normal crossing pair. 
For the details, see \cite[4.7.~Remark]{fujino-fujisawa-saito}. 
If $f\colon X\to Y$ is projective, then everything works well 
in Saito's usual framework of mixed Hodge modules (see 
\cite{saito1} and \cite{saito2}). 
When $f$ is only a proper K\"ahler morphism 
on each irreducible component of $X$, 
we have to use the decomposition theorem for proper 
K\"ahler morphisms announced in \cite{saito-tohoku}. 
The reader can find some remarks and supplementary results 
in \cite{saito5}. 
Theorem \ref{a-thm2.6} for projective morphisms 
is sufficient for the proof of Theorems \ref{a-thm1.1} and 
\ref{a-thm1.2}. Therefore, the reader who is only interested 
in Theorems \ref{a-thm1.1} and \ref{a-thm1.2} can assume that 
$f$ is projective. 
\end{rem}

The following remark is due to Morihiko Saito. 

\begin{rem}\label{a-rem2.8}
In \cite{fujino-fujisawa-saito}, there are 
some changes of English from the original version 
(see, for example, arXiv:1302.6180v2 [math.AG]) by 
the printer unexpectedly. Some of them are misleading. 
\begin{itemize}
\item
On page 86 from lines 3 to 4, 
the reader can find, \lq\lq There is no shift of one in 
[De2], [SZ] although different notation is used.\rq\rq; however, 
this is misleading. We claim, \lq\lq There is no shift 
{\em{compared}} with the one in [De2], [SZ] although 
different notation is used.\rq\rq 
\item 
In 
\cite[4.6.~Remarks.(iv)]{fujino-fujisawa-saito}, the 
reader can find, \lq\lq This can be obtained by 
applying the standard cohomological functors ...\rq\rq; however, 
this is not correct. In the original version 
of \cite{fujino-fujisawa-saito}, we claim, \lq\lq This 
can be obtained by {\em{repeating}} the standard 
cohomological functors ...\rq\rq
\end{itemize}
\end{rem}

We need a special case of Takegoshi's theorem 
(see \cite{takegoshi}), which is a complex analytic generalization of 
Koll\'ar's torsion-freeness and vanishing theorem. 
For the details of Theorem \ref{a-thm2.9}, 
see also \cite[Chapter V.~3.7.Theorem]{nakayama2}. 

\begin{thm}[{see \cite{takegoshi}}]\label{a-thm2.9}
Let $f\colon X\to Y$ be a proper surjective morphism 
from a smooth irreducible complex analytic space $X$ 
such that $X$ is K\"ahler. 
Then $R^qf_*\omega_X$ is torsion-free for every $q$. 

Let $\pi\colon Y\to Z$ be a projective morphism 
between complex analytic spaces and let $\mathcal A$ be a 
$\pi$-ample line bundle on $Y$. 
Then $R^p\pi_*\left(\mathcal A\otimes 
R^qf_*\omega_X\right)=0$ holds for 
every $p>0$ and every $q$. 
\end{thm}

We recall Siu's theorem on complex analytic sheaves, 
which is a special case of \cite[Theorem 4]{siu}. 
We need it in order to state Theorem \ref{a-thm1.1} (i). 

\begin{thm}\label{a-thm2.10} 
Let $\mathcal F$ be a coherent sheaf on a complex 
analytic space $X$. 
Then there exists a locally finite family $\{Y_i\}_{i\in I}$ 
of complex analytic subvarieties of $X$ such that 
\begin{equation*}
\Ass _{\mathcal O_{X,x}}(\mathcal F_x)=\{\mathfrak{p}_{x, 1}, 
\ldots, \mathfrak{p}_{x, r(x)}\}
\end{equation*}
holds for every point $x\in X$, where 
$\mathfrak{p}_{x, 1}, 
\ldots, \mathfrak{p}_{x, r(x)}$ are the prime ideals 
of $\mathcal O_{X, x}$ associated to the irreducible components 
of the germs $Y_{i, x}$ of $Y_i$ at $x$ with $x\in Y_i$. 
We note that each $Y_i$ is called an {\em{associated subvariety}} 
of $\mathcal F$. 
\end{thm}

In this paper, we do not treat the precise definition 
of {\em{proper K\"ahler morphisms of complex analytic spaces}}, which 
is somewhat complicated. Here, we collect all the necessary properties in the following proposition. 
For the details, see, for example, 
\cite[Definition 4.1 and Lemma 4.4]{fujiki}. 

\begin{prop}\label{a-prop2.11} 
Let $f\colon X\to Y$ be a proper morphism 
of complex analytic spaces such that 
$X$ is smooth. 
Then we have: 
\begin{itemize}
\item[(i)] 
If $X$ is K\"ahler, then $f$ is K\"ahler. 
\end{itemize}
On the other hand, if we assume that $f$ is K\"ahler, 
then we have the following properties. 
\begin{itemize}
\item[(ii)] Let $U^\dag$ be any open subset 
of $Y$ such that $U^\dag$ is a closed analytic subspace of 
a polydisc $\Delta^m$ for 
some $m$. 
Then $X_U:=f^{-1}(U)$ is K\"ahler, where $U$ is any relatively compact 
open subset of $U^\dag$. 
\item[(iii)] Let $\pi\colon Y\to Z$ be a projective morphism 
of complex analytic spaces and let $V^\dag$ be any 
open subset of $Z$ such that $V^\dag$ is a closed 
analytic subspace of a polydisc $\Delta^m$ for some $m$. 
Then, for any relatively compact 
open subset $V$ of $V^\dag$, 
there exists the following commutative diagram 
\begin{equation*}
\xymatrix{
Y_V\ar[d]_-\pi \ar@{^{(}->}[r]^-\iota& V\times \mathbb P^n\ar[dl]^-{p_1} \\ 
V 
}
\end{equation*}
where $Y_V:=\pi^{-1}(V)$, $\iota$ is a closed embedding, and $p_1$ 
is the first projection, and $X_V:=(\pi\circ f)^{-1}(V)$ is K\"ahler. 
\end{itemize}
\end{prop}

We will freely use Proposition \ref{a-prop2.11} in the 
subsequent sections. Let us recall some standard definitions. 

\begin{defn}\label{a-def2.11}
Let $D$ be an {\em{$\mathbb R$-divisor}} (resp.~{\em{$\mathbb Q$-divisor}}) 
on an equidimensional complex analytic space $X$. 
This means that $D$ is a locally finite formal sum 
$\sum _i a_i D_i$, where $D_i$ is an irreducible and reduced 
closed analytic subspace of $X$ of codimension one 
with $D_i\ne D_j$ for $i\ne j$ and 
$a_i\in \mathbb R$ (resp.~$a_i\in \mathbb Q$) 
for every $i$. 
We note that the {\em{support}} $\Supp D:=\bigcup _{a_i\ne 0} D_i$ 
of $D$ is 
a closed analytic subspace of $X$. 
If $0\leq a_i \leq 1$ holds for every $i$, then $D$ is called a {\em{boundary}} 
$\mathbb R$-divisor (resp.~$\mathbb Q$-divisor). 
We put 
\begin{equation*}
\lfloor D\rfloor =\sum _i \lfloor a_i\rfloor D_i, \quad 
\lceil D\rceil :=-\lfloor -D\rfloor, \quad \text{and}\quad
\{D\}:=D-\lfloor D\rfloor
\end{equation*}
as usual, where $\lfloor a_i\rfloor$ is the integer defined by $a_i-1<\lfloor a_i 
\rfloor \leq a_i$. 
We also put 
\begin{equation*}
D^{<1}:=\sum _{a_i<1} a_i D_i. 
\end{equation*}
\end{defn}

We need the following special case of the vanishing theorem 
of Reid--Fukuda type, which is an easy consequence of 
the Kawamata--Viehweg vanishing theorem for projective morphisms 
between complex analytic spaces. 

\begin{lem}[Vanishing lemma of Reid--Fukuda type]\label{a-lem2.13}
Let $(X, \Delta)$ be a simple normal crossing pair 
such that $\Delta$ is a boundary 
$\mathbb R$-divisor and 
let $\mathcal L$ be a line bundle on $X$ such that 
$\mathcal L-(\omega_X+\Delta)$ is nef and log big over $Y$ with 
respect to $(X, \Delta)$, 
where $f\colon X\to Y$ is a projective morphism between 
complex analytic spaces. 
Then $R^if_*\mathcal L=0$ holds for every $i>0$. 
\end{lem} 

Before we prove Lemma \ref{a-lem2.13}, we make an 
important remark. 

\begin{rem}\label{a-rem2.14}
In Lemma \ref{a-lem2.13}, 
we always assume that $\Delta$ is globally $\mathbb R$-Cartier 
as in Theorems \ref{a-thm1.1} and 
\ref{a-thm1.2} (see Remark \ref{a-rem1.3}). 
We take an arbitrary point $y\in Y$. 
Then it is sufficient to prove $R^if_*\mathcal L=0$ on 
a small open neighborhood of $y\in Y$. 
Hence, by shrinking $Y$ around $y$ suitably, 
we may assume that there exist Cartier divisors $L$ and $K_X$ on 
$X$ 
such that $\mathcal L\simeq \mathcal O_X(L)$ and 
$\omega_X\simeq \mathcal O_X(K_X)$ hold 
in the proof of Lemma \ref{a-lem2.13}. 
\end{rem}

Let us prove Lemma \ref{a-lem2.13}. 

\begin{proof}[{Proof of Lemma \ref{a-lem2.13}}]
In Step \ref{a-lem2.13-step1}, we will treat the case where 
$X$ is irreducible. Then, in Step \ref{a-lem2.13-step2}, we 
will treat the general case. 
\setcounter{step}{0}
\begin{step}\label{a-lem2.13-step1}
In this step, we assume that $X$ is irreducible. 
If $\lfloor \Delta\rfloor=0$, then $R^if_*\mathcal O_X(L)=0$ holds 
for every $i>0$ by the 
Kawamata--Viehweg vanishing theorem in the complex analytic setting 
(see, for example, \cite[Chapter II.~5.12.Corollary]{nakayama2} 
and \cite[Section 5]{fujino-minimal}).  
From now on, we assume that $\lfloor \Delta\rfloor \ne 0$. 
We take an irreducible component $S$ of $\lfloor \Delta\rfloor$. 
We consider the following short exact sequence: 
\begin{equation*}
0\to \mathcal O_X(L-S)\to \mathcal O_X(L)\to 
\mathcal O_S(L)\to 0. 
\end{equation*}
Note that 
\begin{equation*}
(L-S)-(K_X+\Delta-S)=L-(K_X+\Delta) 
\end{equation*} 
and 
\begin{equation*} 
L|_S-(K_S+(\Delta-S)|_S)=(L-(K_X+\Delta))|_S. 
\end{equation*} 
We use induction on the number of 
irreducible components of $\lfloor \Delta\rfloor$ and on dimension of $X$. 
Thus, we have $R^if_*\mathcal O_X(L-S)=0$ for 
every $i>0$ and $R^if_*\mathcal O_S(L)=0$ for 
every $i>0$. 
Hence $R^if_*\mathcal O_X(L)=0$ for every $i>0$. 
\end{step}
\begin{step}\label{a-lem2.13-step2}
In this step, we use induction on the number of 
irreducible components of $X$. 
Let $Z$ be an irreducible component of $X$. 
Let $Z'$ be the union of components of $X$ other than 
$Z$. 
Then we have the following short exact sequence:  
\begin{equation*}
0\to \mathcal O_{Z'}(L|_{Z'}-Z|_{Z'})\to 
\mathcal O_X(L)\to 
\mathcal O_Z(L|_Z)\to 0.
\end{equation*} 
Note that 
\begin{equation*}
(L|_{Z'}-Z|_{Z'})-(K_{Z'}+\Delta|_{Z'})=(L-(K_X+\Delta))|_{Z'}
\end{equation*}
holds. Thus, by induction, $R^if_*\mathcal O_{Z'} 
(L|_{Z'}-Z|_{Z'})=0$ for every $i>0$. 
Since 
\begin{equation*}
L|_Z-(K_Z+Z'|_Z+\Delta|_Z)=(L-(K_X+\Delta))|_Z
\end{equation*} 
and $Z$ is irreducible, $R^if_*\mathcal O_Z(L|_Z)=0$ for 
every $i>0$ by Step \ref{a-lem2.13-step1}. 
Hence, $R^if_*\mathcal O_Z(L)=0$ holds for every $i>0$. 
\end{step}
We finish the proof. 
\end{proof}

As an easy application of Lemma \ref{a-lem2.13}, 
we obtain a useful lemma. 
This lemma is very useful and indispensable when we 
treat analytic simple normal crossing pairs. 

\begin{lem}\label{a-lem2.15}
Let $g\colon X'\to X$ be a projective bimeromorphic 
morphism between complex analytic spaces such that 
$X$ and $X'$ are simple normal crossing. 
Assume that there exists a Zariski open subset $U$ of 
$X$ such that $g\colon U':=g^{-1}(U)\to U$ is an isomorphism 
and that $U$ 
{\em{(}}resp.~$U'${\em{)}} intersects every stratum of 
$X$ {\em{(}}resp.~$X'${\em{)}}. 
Then $R^ig_*\mathcal O_{X'}=0$ for every $i>0$ and 
$g_*\mathcal O_{X'}\simeq \mathcal O_X$ holds. 
\end{lem}
\begin{proof}
By assumption, we can see that $g$ has connected fibers. 
Hence we may assume that $\codim_X(X\setminus U)\geq 2$. 
Thus, we have $g_*\mathcal O_{X'}\simeq \mathcal O_X$ since $X$ 
satisfies 
Serre's $S_2$ condition. 
Since the problem is local, we can freely shrink 
$X$ throughout this proof.  
We can write $K_{X'}=g^*K_X+E$ for some 
effective $g$-exceptional Cartier divisor 
$E$ on $X'$, 
where $K_X$ (resp.~$K_{X'}$) is a Cartier divisor 
on $X$ (resp.~$X'$) with 
$\mathcal O_X(K_X)\simeq \omega_X$ (resp.~$\mathcal O_{X'}(K_{X'})
\simeq \omega_{X'}$). 
Therefore, we have $g_*\mathcal O_{X'}(K_{X'})\simeq 
\mathcal O_X(K_X)$. 
Since $K_{X'}-K_{X'}=0$, $R^ig_*\mathcal O_{X'}(K_{X'})=0$ holds 
for every $i>0$ by 
Lemma \ref{a-lem2.13}. 
This implies $\mathbf Rg_*\omega^{\bullet}_{X'}\simeq \omega^{\bullet}_X$, 
where $\omega^{\bullet}_X$ (resp.~$\omega^{\bullet}_{X'}$) 
is a dualizing complex of $X$ (resp.~$X'$). 
By Grothendieck duality (see \cite{rrv}), 
we have 
\begin{equation*}
\begin{split}
\mathcal O_X&\simeq \mathbf R\mathcal{H}om (\omega^\bullet_X, 
\omega^\bullet_X)\\ 
&\simeq \mathbf R\mathcal{H}om(\mathbf Rg_*\omega^{\bullet}_{X'}, 
\omega^\bullet_X)\\ 
&\simeq \mathbf Rg_*\mathbf R\mathcal{H}om (\omega^{\bullet}_{X'}, 
\omega^{\bullet}_{X'})\\ 
&\simeq \mathbf Rg_*\mathcal O_{X'}. 
\end{split}
\end{equation*}
This implies the desired statement. 
\end{proof}

In the proof of Theorem \ref{a-thm1.1}, 
we will use Kawamata's covering trick. 
We contain the precise statement here for the sake of completeness. 

\begin{lem}[Kawamata's covering trick]\label{a-lem2.16}
Let $f\colon X\to Y$ be a projective morphism from a smooth 
complex analytic space $X$ onto a Stein space $Y$. 
Let $\Sigma$ be a reduced simple normal crossing divisor on $X$ and 
let $N$ be a $\mathbb Q$-Cartier $\mathbb Q$-divisor 
on $X$ such that $\Supp \{N\}$ and $\Sigma$ have 
no common irreducible components and that 
the support of $\{N\}+\Sigma$ is a simple normal crossing 
divisor on $X$. 
Then, after replacing $Y$ with any relatively compact open subset 
of $Y$, 
we can construct a finite cover $p\colon V\to X$ such that 
$V$ is smooth, $\Sigma_V:=p^*\Sigma$ 
is a reduced simple normal crossing divisor on $V$, 
$N_{V}:=p^*N$ is a Cartier divisor on $V$, and $\mathcal 
O_X(K_X+\Sigma+\lceil N\rceil)$ is a direct summand of 
$p_*\mathcal O_V(K_V+\Sigma_V+N_V)$. 
\end{lem}

\begin{proof}
The usual proof of Kawamata's covering trick (see, for example, 
\cite[3.19. Lemma]{esnault-viehweg}) can work in the 
above complex analytic setting with only minor modifications. 
\end{proof}

In the proof of Theorem \ref{a-thm1.2}, we need the 
following elementary lemma. 
We describe it here for the sake of completeness. 

\begin{lem}\label{a-lem2.17}
Let $(X, \Delta)$ be an analytic globally embedded simple 
normal crossing pair and let $M$ be the ambient space 
of $(X, \Delta)$ such that 
$\Delta$ is a boundary $\mathbb R$-divisor 
{\em{(}}resp.~$\mathbb Q$-divisor{\em{)}}. 
Let $C$ be a stratum of $(X, \Delta)$, which 
is not an irreducible component of $X$. 
Let $\sigma\colon M'\to M$ be the blow-up along $C$ and let 
$X'$ denote the reduced structure of the total transform of 
$X$ on $M'$, that is, $X'=\sigma^{-1}(X)$. 
We put 
\begin{equation*}
K_{X'}+\Delta':=g^*(K_X+\Delta), 
\end{equation*} 
where $g:=\sigma|_{X'}$. Then we have the following properties: 
\begin{itemize}
\item[(i)] $(X', \Delta')$ is an analytic globally embedded simple 
normal crossing pair such that $\Delta'$ is a boundary 
$\mathbb R$-divisor {\em{(}}resp.~$\mathbb Q$-divisor{\em{)}}, 
\item[(ii)] $M'$ is the ambient space of $(X', \Delta')$, 
\item[(iii)] $g_*\mathcal O_{X'}\simeq \mathcal O_X$ holds 
and $R^ig_*\mathcal O_{X'}=0$ for every $i>0$, 
\item[(iv)] the strata of $(X, \Delta)$ are exactly the images of 
the strata of $(X', \Delta')$, 
and 
\item[(v)] $\sigma^{-1}(C)$ is a maximal 
{\em{(}}with respect to the inclusion{\em{)}} stratum of $(X', \Delta')$. 
\end{itemize}
\end{lem}

\begin{proof}
We can write $\Delta=B|_X$ by definition. 
Then we have 
\begin{equation*}
K_{M'}+X'+B'=\sigma^*(K_M+X+B), 
\end{equation*} 
where $B'$ is the strict transform of $B$ on $M'$. 
We put $c=\codim _MC\geq 2$. 
Then we obtain $\sigma ^* X=X'+(k-1)E$ with $1\leq k\leq c$ and 
\begin{equation*}
K_{M'}=\sigma^*K_M+(c-1)E, 
\end{equation*} 
where $E$ is the $\sigma$-exceptional divisor on $M'$. 
We consider the following short exact sequence: 
\begin{equation*}
0\to \mathcal O_{M'}(-X')\to \mathcal O_{M'}\to \mathcal O_{X'}\to 0. 
\end{equation*} 
Since $M$ is smooth, we have $R^i\sigma_*\mathcal O_{M'}=0$ for 
every $i>0$. 
Since 
\begin{equation*}
-X'-K_{M'}=-\sigma^*(K_M+X)-(c-k)E, 
\end{equation*}
we have $R^i\sigma_*\mathcal O_{M'}(-X')=0$ for every $i>0$ by 
the Kawamata--Viehweg vanishing theorem in the complex analytic 
setting (see, for example, \cite[Chapter II.~5.12.Corollary]{nakayama2} 
and \cite[Section 5]{fujino-minimal}). 
Thus, we have $R^ig_*\mathcal O_{X'}=0$ for every $i>0$ and the following 
short exact sequence: 
\begin{equation*}
0\to \sigma_*\mathcal O_{M'}(-X')=\mathcal I_X\to 
\mathcal O_M\to g_*\mathcal O_{X'}\to 0, 
\end{equation*} 
where $\mathcal I_X$ is the defining ideal sheaf of $X$ on $M$. 
This implies that $g_*\mathcal O_{X'}\simeq \mathcal O_X$ holds. 
By construction, $\Delta'=B'|_{X'}$ holds. 
Therefore, $(X', \Delta')$ is an analytic globally embedded simple 
normal crossing pair such that $M'$ is the ambient space of $(X', \Delta')$. 
We can easily check that $(X', \Delta')$ satisfies all the desired properties. 
\end{proof}

\section{Standard setting}\label{a-sec3}

In this section, we will establish the following theorem, which 
obviously generalizes Koll\'ar's famous torsion-freeness, 
vanishing theorem, and injectivity theorem. 
Although our result in this section depends on 
Saito's theory of mixed Hodge modules through 
Theorem \ref{a-thm2.6} (see 
\cite{saito1}, \cite{saito2}, \cite{saito-tohoku}, 
\cite{fujino-fujisawa-saito}, and \cite{saito5}), 
we do not directly use Saito's vanishing theorem in \cite{saito2}. 

\begin{thm}\label{a-thm3.1}
Let $(X, D)$ be an analytic simple normal crossing pair 
such that $D$ is reduced and 
let $f\colon X\to Y$ be a proper 
morphism between complex analytic spaces. 
We assume that $f$ is K\"ahler 
on each irreducible component of $X$. 
Then we have the following properties. 
\begin{itemize}
\item[(i)] $($Strict support condition$)$.~Every 
associated subvariety of $R^qf_*\omega_X(D)$ 
is the $f$-image of some stratum of $(X, D)$ for every $q$. 
\item[(ii)] $($Vanishing theorem$)$.~Let $\pi\colon Y\to Z$ be 
a projective morphism between complex analytic 
spaces and let $\mathcal A$ be a $\pi$-ample 
line bundle on $Y$. 
Then 
\begin{equation*}
R^p\pi_*\left(\mathcal A\otimes R^qf_*\omega_X(D)\right)=0
\end{equation*} 
holds  
for every $p>0$ and every $q$. 
\item[(iii)] $($Injectivity theorem$)$.~Let $\mathcal L$ be an $f$-semiample 
line bundle on $X$. Let $s$ be a nonzero element of $H^0(X, 
\mathcal L^{\otimes k})$ for some nonnegative integer $k$ such that 
the zero locus of $s$ does not contain any strata of $(X, D)$. 
Then, for every $q$, 
the map 
\begin{equation*}
\times s\colon R^qf_*\left(\omega_X(D)\otimes \mathcal L^{\otimes l}\right)
\to R^qf_*\left(\omega_X(D)\otimes \mathcal L^{\otimes k+l}\right)
\end{equation*} 
induced by $\otimes s$ is injective for every positive integer $l$. 
\end{itemize}
\end{thm}

As we have already mentioned above, Theorem \ref{a-thm3.1} 
generalizes Koll\'ar's famous results. 
We explain it here for the reader's convenience. 

\begin{rem}[Koll\'ar's original theorem]\label{a-rem3.2}
If $X$ is a smooth projective variety with $D=0$ and 
$f\colon X\to Y$ is a projective surjective 
morphism onto a projective variety $Y$ in 
Theorem \ref{a-thm3.1} (i), then 
the strict support condition is nothing but 
Koll\'ar's torsion-freeness of $R^qf_*\omega_X$ 
(see \cite[Theorem 2.1 (i)]{kollar1}). 
We further assume that $Z$ is a point in Theorem \ref{a-thm3.1} (ii). 
Then we can recover Koll\'ar's vanishing theorem 
(see \cite[Theorem 2.1 (iii)]{kollar1}). 
If $X$ is a smooth projective variety, 
$D=0$, and $Y$ is a point, then Theorem \ref{a-thm3.1} (iii) 
coincides with Koll\'ar's original injectivity theorem 
(see \cite[Theorem 2.2]{kollar1}). 
\end{rem}

\begin{rem}\label{a-rem3.3}
Theorem \ref{a-thm3.1} (iii) solves \cite[Problem 1.8]{fujino-trans2} 
completely. Note that \cite[Conjecture 2.21]{fujino-on-semipositivity} 
is closely related to Theorem \ref{a-thm3.1} (iii) and 
was recently solved in \cite{cao-paun} (see also \cite{chan-matsumura}). 
\end{rem}

Let us start the proof of Theorem \ref{a-thm3.1}. 

\begin{proof}[Proof of Theorem \ref{a-thm3.1}]
In Step \ref{a-thm3.1-step1}, we will prove (i), which is an 
easy consequence of Theorem \ref{a-thm2.6} and 
Takegoshi's torsion-freeness (see Theorem \ref{a-thm2.9}). 
In Step \ref{a-thm3.1-step2}, we will prove that 
(ii) easily follows from Takegoshi's vanishing theorem (see Theorem 
\ref{a-thm2.9}) 
with the aid of Theorem \ref{a-thm2.6}. 
In Step \ref{a-thm3.1-step3}, we will see that (iii) is 
an easy consequence of (i) and (ii). 
\setcounter{step}{0}
\begin{step}[Strict support condition]\label{a-thm3.1-step1}
Since the problem is local, we may assume that 
$Y$ is a closed analytic subspace of a polydisc $\Delta^m$. 
By replacing $Y$ with $\Delta^m$, 
we may further assume that $Y$ itself is a polydisc 
(see also Proposition \ref{a-prop2.11} (ii)). 
In this case, we can use Theorem \ref{a-thm2.6}. 
We note that $\omega_Y\simeq \mathcal O_Y$ holds. 
By Theorem \ref{a-thm2.9}, 
${}_FE^{-q, i+q}_1\simeq \bigoplus _{k+l=n+q+1}R^if_*\omega_{D^{[k, l]}}$
satisfies the strict support condition, that is, 
every associated subvariety of 
${}_FE^{-q, i+q}_1\simeq \bigoplus _{k+l=n+q+1}R^if_*\omega_{D^{[k, l]}}$ 
is the $f$-image of some stratum of $(X, D)$. 
By Theorem \ref{a-thm2.6}, the associated subvariety 
of ${}_FE^{-q, i+q}_2={}_FE^{-q, i+q}_\infty$ 
is the $f$-image of some stratum of $(X, D)$. 
This implies that $R^qf_*\omega_X(D)$ satisfies the 
desired strict support condition. 
\end{step}
\begin{step}[Vanishing theorem]\label{a-thm3.1-step2}
We may assume that $Z$ is a polydisc and $Y$ is a closed analytic 
subspace of $Z\times \mathbb P^n$ 
(see Proposition \ref{a-prop2.11} (iii)). 
By applying Theorem \ref{a-thm2.6} to 
$f\colon X\to Y\hookrightarrow Z\times \mathbb P^n$, 
we obtain the following spectral sequence 
\begin{equation*}
E^{-q, i+q}_1=\bigoplus _{k+l=n+q+1}R^if_*\omega_{D^{[k, l]}} 
\Rightarrow R^if_*\omega_X(D)
\end{equation*}
which degenerates at $E_2$ such that its $E_1$-differential $d_1$ splits. 
By Theorem \ref{a-thm2.9}, 
we obtain 
$R^p\pi_*\left(\mathcal A\otimes E^{-q, i+q}_1\right)=0$ for every 
$p>0$. Since the $E^{-q, i+q}_2=E^{-q, i+q}_\infty$ are direct factors of 
$E^{-q, i+q}_1$, we have 
$R^p\pi_*\left(\mathcal A\otimes E^{-q, i+q}_2\right)=0$ 
for every $p>0$. 
This implies that $R^p\pi_*\left(\mathcal A\otimes R^qf_*\omega_X(D)\right)=0$ 
holds for every $p>0$. 
This is what we wanted. 
\end{step}
\begin{step}[Injectivity theorem]\label{a-thm3.1-step3} 
We take an arbitrary point $y\in Y$ and can freely 
shrink $Y$ around $y$ 
since the problem is local. 
Without loss of generality, we may assume that 
$Y$ is Stein and that each irreducible component of $X$ is K\"ahler 
(see Proposition \ref{a-prop2.11} (ii)). 
By Bertini's theorem, we can take an element 
$u\in H^0(X, \mathcal L^{\otimes m})\setminus \{0\}$ for 
some positive integer $m\geq 2$ such that 
$R:=(u=0)$ and $R+D$ are reduced and $(X, D+R)$ is 
an analytic simple 
normal crossing pair. 
Let $p\colon V\to X$ be the $m$-fold cyclic cover 
ramifying along $R$. 
More precisely, we define an $\mathcal O_X$-algebra structure of 
$\bigoplus _{i=0}^{m-1}\mathcal L^{\otimes (-i)}$ by 
$u\colon \mathcal L^{\otimes (-m)}\to \mathcal O_X$ and 
put 
\begin{equation*}
V:=\Specan_X\bigoplus _{i=0}^{m-1}\mathcal L^{\otimes (-i)}. 
\end{equation*}  
Then $(V, p^*D)$ is an analytic simple normal crossing pair. 
We can check that each irreducible component of 
$V$ is K\"ahler if we shrink $Y$ around $y$ suitably.  
By construction, we have 
\begin{equation*}
p_*\left(\omega_V(p^*D)\otimes p^*\mathcal L^{\otimes l}\right)
=\bigoplus _{i=l}^{l+m-1} \omega_X(D)\otimes \mathcal L^{\otimes i}. 
\end{equation*} 
By construction again, 
we see that $p^*\mathcal L$ has a section whose 
zero locus is the reduced preimage  of $R$. 
By iterating this process, we obtain a tower 
of cyclic covers: 
\begin{equation*}
V_n\to V_{n-1}\to \cdots \to V_0:=V\to X. 
\end{equation*} 
By choosing the ramification divisors suitably, we may assume 
that 
the pull-back of $\mathcal L$ on $V_n$ is relatively globally 
generated over $Y$. By replacing $X$ and $\mathcal L$ with 
$V_n$ and the pull-back of $\mathcal L$, respectively, 
we may assume that $\mathcal L$ is $f$-free, that is, 
$f^*f_*\mathcal L\to \mathcal L$ is surjective. 
We take the contraction morphism over $Y$ associated to 
the surjection $f^*f_*\mathcal L
\to \mathcal L$ and 
take its Stein factorization. 
Then we obtain the following 
commutative diagram
\begin{equation*}
\xymatrix{
X\ar[dr]_-f\ar[rr]^-g& & W\ar[dl]^-h \\ 
& Y&
}
\end{equation*}
such that 
$g_*\mathcal O_X\simeq \mathcal O_W$ and that 
$\mathcal L\simeq g^*\mathcal L_W$ for some 
$h$-ample 
line bundle $\mathcal L_W$ on $W$. 
Since $H^0(X, \mathcal L^{\otimes k})
\simeq 
H^0(W, \mathcal L^{\otimes k}_W)$, 
there exists $t\in H^0(W, \mathcal L^{\otimes k}_W)$ such that 
$s=g^*t$. 
By (i), every associated subvariety 
of $R^qg_*\omega_X(D)$ is the $g$-image 
of some stratum of $(X, D)$. 
By assumption, the zero locus of $t$ does not contain 
any associated subvarieties of $R^qg_*\omega_X(D)$. 
Hence, the map 
\begin{equation*}
\times t\colon R^qg_*\omega_X(D)\otimes \mathcal L^{\otimes l}_W
\to R^qg_*\omega_X(D)\otimes \mathcal L^{\otimes k+l}_W
\end{equation*} 
induced by $\otimes t$ is injective. Thus, the map 
\begin{equation*}
\times s \colon R^qg_*\left(\omega_X(D)\otimes 
\mathcal L^{\otimes l}\right)
\to R^qg_*\left(\omega_X(D)\otimes \mathcal L^{\otimes (k+l)}\right)
\end{equation*} 
is injective. Therefore, by taking $h_*$, 
we obtain that 
\begin{equation}\label{a-eq3.1}
\times s\colon h_*\left(R^qg_*\left(\omega_X(D)\otimes 
\mathcal L^{\otimes l}\right)\right)
\to h_*\left(R^qg_*\left(\omega_X(D)\otimes \mathcal L^{\otimes (k+l)}\right)
\right)
\end{equation} 
is injective. 
On the other hand, 
\begin{equation*}
R^ph_*\left(R^qg_*\left(\omega_X(D)\otimes \mathcal L^{\otimes n}\right)
\right)\simeq R^ph_*\left(\mathcal L^{\otimes n}_W\otimes 
R^qg_*\omega _X(D)\right)=0
\end{equation*} 
holds for every $p>0$ and every $n>0$ by (ii). 
Thus, by using the spectral sequence, 
we have 
\begin{equation*}
h_*R^qg_*\left(\omega_X(D)\otimes \mathcal L^{\otimes n}\right)
\simeq R^q(h\circ g)_*\left(\omega_X(D)\otimes \mathcal L^{\otimes n}\right)
=R^qf_*\left(\omega_X(D)\otimes \mathcal L^{\otimes n}\right)
\end{equation*} 
for every positive integer $n$. 
Therefore, \eqref{a-eq3.1} implies the desired injection 
\begin{equation*}
\times s\colon R^qf_*\left(\omega_X(D)
\otimes \mathcal L^{\otimes l}\right)\to 
R^qf_*\left (\omega_X(D)\otimes \mathcal L^{\otimes (k+l)}\right). 
\end{equation*} 
\end{step}
We finish the proof. 
\end{proof}

\section{Proof of Theorem \ref{a-thm1.1} when $X$ is irreducible}\label{a-sec4}

In this section, we will prove Theorem \ref{a-thm1.1} under the extra 
assumption that $X$ is irreducible. 
For many geometric applications, the case where $X$ is irreducible 
seems to be sufficient. 
When $X$ is irreducible, we can easily reduce Theorem \ref{a-thm1.1} 
to Theorem \ref{a-thm3.1} (i) and (ii) 
by using some standard arguments which are 
repeatedly used in the theory of minimal models. 

First, let us prove Theorem \ref{a-thm1.1} (i) when $X$ is irreducible. 
In the proof, we will use some covering tricks to 
reduce it to Theorem \ref{a-thm3.1} (i). 

\begin{proof}[Proof of Theorem \ref{a-thm1.1} (i) when 
$X$ is irreducible] 
We take an arbitrary point $y\in Y$. 
It is sufficient to prove the strict support condition on a small 
Stein open neighborhood of $y$. 
Therefore, we will freely shrink $Y$ around $y$ suitably 
without mentioning it explicitly throughout this proof 
(see also Proposition \ref{a-prop2.11}). 
By replacing $Y$ with a small relatively compact 
Stein open neighborhood 
of $y$, 
we may assume that there exist Cartier divisors $L$ and $K_X$ such that 
$\mathcal O_X(L)\simeq \mathcal L$ and $\mathcal O_X(K_X)\simeq 
\omega_X$, respectively. 
By perturbing the coefficients of $\Delta$ suitably, 
we may further assume that 
$\Delta$ is a $\mathbb Q$-divisor. 
We put $N:=L-(K_X+\Delta)$. 
Then $N$ is an $f$-semiample $\mathbb Q$-Cartier 
$\mathbb Q$-divisor on $X$ such that 
$L=K_X+\Sigma +\lceil N\rceil$, where $\Sigma:=\lfloor 
\Delta\rfloor$, with $\{N\}=\{-\Delta\}$. 
By Kawamata's covering trick (see Lemma \ref{a-lem2.16}), we can take a finite 
cover $p\colon V\to X$ such that 
$V$ is smooth, 
$\Sigma_V:=p^*\Sigma$ is a reduced simple normal crossing divisor 
on $V$, $N_V:=p^*N$ is Cartier, and 
$\mathcal O_X(L)=\mathcal O_X(K_X+\Sigma +\lceil N\rceil)$ is 
a direct summand of $p_*\mathcal O_V(K_V+\Sigma_V+N_V)$. Since 
$N_V$ is an $(f\circ p)$-semiample Cartier divisor on $V$, 
we can take a finite cyclic cover $q\colon W\to V$ such that  
$W$ is smooth, $\Sigma_W:=q^*\Sigma_V=q^*p^*\Sigma$ is a reduced 
simple normal crossing divisor on $W$, and 
$\mathcal O_V(K_V+\Sigma_V+N_V)$ is a direct summand 
of $q_*\mathcal O_W(K_W+\Sigma_W)$. 
By Theorem \ref{a-thm3.1} (i), 
every associated subvariety of $R^q(f\circ p\circ q)_*\mathcal 
O_W(K_W+\Sigma_W)$ 
is the $(f\circ p\circ q)$-image of some stratum of $(W, \Sigma_W)$. 
This implies that every associated subvariety of $R^qf_*\mathcal O_X(L)$ 
is the $f$-image of some stratum of $(X, \Delta)$. 
This is what we wanted. 
\end{proof}

Next, we will use some covering tricks and Leray's 
spectral sequence in order to reduce Theorem \ref{a-thm1.1} (ii) 
to Theorem \ref{a-thm3.1} (ii). 
Although the covering tricks are the same as above, 
we will write all the details for the reader's convenience. 

\begin{proof}[Proof of Theorem \ref{a-thm1.1} (ii) when $X$ is irreducible]
As in the proof of Theorem \ref{a-thm1.1} (i), we take an arbitrary 
point $z\in Z$ and consider a small relatively compact Stein 
open neighborhood of $z$ (see also Proposition \ref{a-prop2.11}). 
Throughout this proof, we will freely shrink $Z$ around $z$ 
without mentioning it explicitly. 
By perturbing the coefficients, 
we may assume that $H$ is a $\pi$-ample 
$\mathbb Q$-divisor on $Y$ with 
$L-(K_X+\Delta)\sim _{\mathbb Q} f^*H$, 
where $\mathcal L\simeq \mathcal O_X(L)$ and $\omega_X\simeq 
\mathcal O_X(K_X)$. 
We put $N:=L-(K_X+\Delta)$ and $\Sigma:=\lfloor 
\Delta\rfloor$. 
By Kawamata's covering trick (see Lemma \ref{a-lem2.16}), 
we can construct a finite cover $p\colon V\to X$ such that 
$V$ is smooth, $\Sigma_V:=p^*\Sigma$ is a simple 
normal crossing divisor, and $N_V:=p^*N$ is a Cartier 
divisor on $V$. 
Since $\mathcal O_X(L)=\mathcal O_X(K_X+\Sigma+\lceil N\rceil)$ is 
a direct summand of $p_*\mathcal O_V(K_V+\Sigma_V+N_V)$, 
it is sufficient to prove that 
\begin{equation*}
R^p\pi_*R^q(f\circ p)_*\mathcal O_V(K_V+\Sigma_V+N_V) =0
\end{equation*} 
for every $p>0$. 
As in Step \ref{a-thm3.1-step3} in the proof of 
Theorem \ref{a-thm3.1}, we can 
take a finite cover $q\colon W\to V$ such that 
$W$ is smooth, 
$N_W:=q^*N_V$ is $(\pi\circ f\circ p\circ q)$-free, 
$\Sigma_W:=q^*\Sigma_V$ is a reduced simple normal crossing 
divisor, and $\mathcal O_V(K_V+\Sigma_V+N_V)$ 
is a direct summand of $q_*\mathcal O_W(K_W+\Sigma_W+N_W)$. 
Hence it is sufficient to prove that 
\begin{equation*}
R^p\pi_*R^q(f\circ p\circ q)_*\mathcal O_W(K_W+\Sigma_W+N_W)=0
\end{equation*} 
for every $p>0$. 
We take the contraction morphism 
over $Z$ associated to the 
surjection $(\pi\circ f\circ p\circ q)^*(\pi\circ f\circ p\circ q)_*
\mathcal O_W(N_W)\to \mathcal O_W(N_W)$ and take its Stein 
factorization. 
Then we have the following commutative diagram: 
\begin{equation*}
\xymatrix{
W \ar[dr]_-{f\circ p\circ q}\ar[rr]^-g&& Y^\dag\ar[dl]^-h \\ 
& Y \ar[d]^-\pi&\\ 
&Z&
}
\end{equation*} 
such that $g_*\mathcal O_W\simeq \mathcal O_{Y^\dag}$, 
$\mathcal O_W(N_W)\simeq g^*\mathcal A$, where 
$\mathcal A$ is a $(\pi\circ h)$-ample line bundle 
on $Y^\dag$, and $h$ is finite. 
Then we obtain 
\begin{equation*}
\begin{split}
&R^p\pi_*R^q(f\circ p\circ q)_*\mathcal O_W(K_W+\Sigma_W+N_W)\\ 
&\simeq R^p\pi_*R^q(h\circ g)_*\mathcal O_W(K_W+\Sigma_W+N_W)\\ 
&\simeq R^p\pi_*\left(h_*R^qg_*\mathcal O_W(K_W+\Sigma_W+N_W)\right)\\ 
&\simeq R^p(\pi\circ h)_*
\left(\mathcal A\otimes R^qg_*\mathcal O_W(K_W+\Sigma_W)\right)
\\ 
&=0
\end{split} 
\end{equation*}
by Theorem \ref{a-thm3.1} (ii). 
We finish the proof. 
\end{proof}

\section{Proof of Theorem \ref{a-thm1.1}}\label{a-sec5}

In this section, we will prove Theorem \ref{a-thm1.1} in full generality. 
The proof of Theorem \ref{a-thm1.1} given in this section is essentially 
the same as that of Theorem \ref{a-thm1.1} for irreducible varieties 
in Section \ref{a-sec4} although it is more complicated. 

First, we prepare the following technical but very important 
lemma. 

\begin{lem}\label{a-lem5.1}
Let $(X, \Delta)$ be an analytic 
simple normal crossing pair such 
that $\Delta$ is a boundary $\mathbb R$-divisor 
{\em{(}}resp.~$\mathbb Q$-divisor{\em{)}} and let $f\colon X\to Y$ be a projective 
morphism between complex analytic spaces. 
Let $L$ be a Cartier divisor on $X$. 
We take an arbitrary point $y\in Y$. 
Then, after shrinking $Y$ around $y$ suitably, 
we can construct the following commutative 
diagram: 
\begin{equation*}
\xymatrix{
Z\ar[d]_-p\ar@{^{(}->}[r]^-\iota& M\ar[dd]^-q\\ 
X\ar[d]_f \\ 
Y\ar@{^{(}->}[r]_-{\iota_Y}& \Delta^m
}
\end{equation*}
such that 
\begin{itemize}
\item[(i)] $\iota_Y\colon Y\hookrightarrow \Delta^m$ is a 
closed embedding into a polydisc $\Delta^m$, 
\item[(ii)] $(Z, \Delta_Z)$ is an analytic globally embedded simple 
normal crossing pair, 
where $\Delta_Z$ is a boundary $\mathbb R$-divisor 
{\em{(}}resp.~$\mathbb Q$-divisor{\em{)}} on $Z$, 
\item[(iii)] $M$ is the ambient space of $(Z, \Delta_Z)$ and 
is projective over $\Delta^m$, 
\item[(iv)] there exists a Cartier divisor $L_Z$ on $Z$ 
satisfying 
\begin{equation*}
L_Z-(K_Z+\Delta_Z)=p^*(L-(K_X+\Delta)), 
\end{equation*} 
$p_*\mathcal O_Z(L_Z)\simeq \mathcal O_X(L)$, and 
$R^ip_*\mathcal O_Z(L_Z)=0$ for 
every $i>0$, 
\item[(v)] $p(W)$ is a stratum of $(X, \Delta)$ for every 
stratum $W$ of $(Z, \Delta_Z)$, 
\item[(vi)] there exists a Zariski open subset $U$ of $X$, which 
intersects every stratum of $X$, such that 
$p$ is an isomorphism over $U$,  
\item[(vii)] $p$ maps every stratum of $Z$ 
bimeromorphically onto some stratum of $X$, and 
\item[(viii)] for any stratum $S$ of $(X, \Delta)$, 
there exists a stratum $W$ of $(Z, \Delta_Z)$ such that 
$S=p(W)$.   
\end{itemize}
\end{lem}

\begin{proof}
We divide the proof into several small steps. 
The arguments below are essentially contained in 
\cite[Section 4]{fujino-fundamental-slc}, 
\cite[Section 5.8]{fujino-foundations}, and 
\cite[Lemmas 4.4, 4.6, and 4.8]{fujino-pull}. 
\setcounter{step}{0}
\begin{step}\label{a-lem5.1-step1}Since $f\colon X\to Y$ is projective, we have 
the following commutative diagram 
\begin{equation*}
\xymatrix{
X\ar[d]_-f\ar@{^{(}->}[r]& Y\times \mathbb P^n\ar[dl]\\ 
Y & 
}
\end{equation*}
after shrinking $Y$ around $y$. 
Without loss of generality, we may further assume that 
$Y$ is a closed analytic subspace 
of a polydisc $\Delta^m$. 
Thus we get the following commutative diagram 
\begin{equation*}
\xymatrix{
X\ar[d]_-f\ar@{^{(}->}[r]& \Delta^m\times \mathbb P^n\ar[d]\\
Y\ar@{^{(}->}[r]_-{\iota_Y} & \Delta^m
}
\end{equation*}
such that $\iota_Y(y)=0\in \Delta^m$. 
\end{step}
\begin{step}\label{a-lem5.1-step2}
We put $M_0:=\Delta^m\times \mathbb P^n$. 
We take an irreducible component $C$ of $\Supp \Delta$ and 
take the blow-up $p_1\colon M_1\to M_0$ along $C$. 
Let $X_1$ be the strict transform of $X$ on $M_1$. We 
put $K_{X_1}+\Delta_1=p^*_1(K_X+\Delta)$ and $L_1=p_1^*L$. 
By construction, $\Delta_1$ is a boundary $\mathbb R$-divisor 
(resp.~$\mathbb Q$-divisor). 
By Lemma \ref{a-lem2.15}, 
$R^ip_{1*}\mathcal O_{X_1}=0$ for every $i>0$ and 
$p_{1*}\mathcal O_{X_1}\simeq 
\mathcal O_X$ holds. 
Therefore, we obtain that 
$R^ip_{1*}\mathcal O_{X_1}(L_1)=0$ for 
every $i>0$ and $p_{1*}\mathcal O_{X_1}(L_1)\simeq \mathcal O_X(L)$. 
Note that $M_1$ is smooth and 
that $(X_1, \Delta_1)$ is a simple 
normal crossing pair. 
By repeating this process finitely many times, 
we obtain a sequence of blow-ups 
\begin{equation*}
\xymatrix{
M_k \ar[r]^-{p_k}& M_{k-1}\ar[r]^-{p_{k-1}} & \cdots \ar[r]^-{p_1}& M_0
}
\end{equation*}
and simple normal crossing pairs $(X_i, \Delta_i)$ and Cartier 
divisors $L_i$ on $X_i$ such that 
$\Delta_k=D|_{X_k}$ for some $\mathbb R$-Cartier $\mathbb R$-divisor 
(resp.~$\mathbb Q$-Cartier $\mathbb Q$-divisor) $D$ 
on $M_k$. 
\end{step}
\begin{step}[{see the proof of 
\cite[Proposition 10.59]{kollar4}}]\label{a-lem5.1-step3}
We shrink $\Delta^m$ slightly and assume that 
$M_k$ is a closed analytic subspace of $\Delta^m\times \mathbb P^{n'}$. 
We pick a finite set of points $W$ of $X_k$ such that 
each stratum of $(X_k, \Supp \Delta_k)$ contains 
some point of $W$. 
We take a sufficiently large positive integer $d$ such that 
$\mathcal I_{X_k}\otimes \mathcal O_{M_k}(d)$ is globally generated, 
where $\mathcal I_{X_k}$ is the defining ideal sheaf of $X_k$ on $M_k$ and 
$\mathcal O_{M_k}(d):=\left(q^*\mathcal O_{\mathbb P^{n'}}(d)\right)|_{M_k}$ 
with the second projection $q\colon \Delta^m\times \mathbb P^{n'}\to 
\mathbb P^{n'}$. 
We take a complete intersection of $(m+n'-\dim X_k-1)$ general 
members of $|\mathcal I_{X_k}\otimes \mathcal O_{M_k}(d)|$. 
Then we can construct $X_k\subset V$ such that 
$V$ is smooth at every point of $W$. 
We note that we used the fact that $X_k$ has only hypersurface singularities 
near $W$. By construction, we have 
$V\not\subset \Supp D$. 
\end{step}
\begin{step}\label{a-lem5.1-step4} 
In this step, we will freely shrink $\Delta^m$ slightly without mentioning it 
explicitly. 
By applying the resolution of singularities to $V$ (see 
\cite[Theorem 13.3]{bierstone-milman}), we can construct 
a projective bimeromorphic morphism 
$\alpha\colon V'\to V$ from a smooth analytic space $V'$. 
We may assume that the exceptional locus $E$ of 
$\alpha$ is a simple normal crossing divisor 
on $V'$ and that $\alpha\colon V'\to V$ is an isomorphism 
over the largest Zariski open subset of $V$ on which $V$ is smooth. 
Let $X_{k+1}$ be the strict transform of $X_k$ on $V'$. 
We apply the resolution of singularities 
to $(V', X_{k+1}+\Sigma)$ (see 
\cite[Theorems 13.3 and 12.4]{bierstone-milman}), 
where $\Sigma$ is the support of the 
union of $\alpha^*(D|_V)$ and $E$. 
Then we get a projective bimeromorphic morphism 
$M\to V'$ from a smooth analytic space $M$ 
such that $M\to V'$ is an isomorphism over 
the locus where $X_{k+1}+\Sigma$ is a simple normal crossing divisor. 
Let $Z$ be the strict transform of $X_{k+1}$ on $M$. 
We put $K_Z+\Theta:=\beta^*(K_{X_k}+\Delta_k)$, where 
$\beta\colon Z\to X_k$ is the natural induced morphism. 
We put $\Delta_Z:=\Theta+\lceil -(\Theta^{<1})\rceil$ and 
$L_Z:=\beta^*L_k+\lceil -(\Theta^{<1})\rceil$. 
By construction, $(Z, \Delta_Z)$ is an analytic globally embedded 
simple normal crossing pair and $M$ is the ambient space of 
$(Z, \Delta_Z)$. 
We note that $\Delta_Z$ is a boundary $\mathbb R$-divisor 
(resp.~$\mathbb Q$-divisor) and that 
\begin{equation*}
\begin{split}
L_Z-(K_Z+\Delta_Z)&=\beta^*L_k+\lceil -(\Theta^{<1})\rceil 
-\beta^*(K_{X_k}+\Delta_k)-\lceil -(\Theta^{<1})\rceil\\ 
&=\beta^*(L_k-(K_{X_k}+\Delta_k))
\end{split} 
\end{equation*} holds. 
Note that $\lceil -(\Theta^{<1})\rceil$ is an effective 
$\beta$-exceptional divisor on $Z$. 
By construction again, $\beta$ maps every stratum $W$ of $(Z, \Delta_Z)$ 
bimeromorphically onto $\beta(W)$. 
Therefore, we have $R^i\beta_*\mathcal O_Z(L_Z)=0$ for 
every $i>0$ by Lemma \ref{a-lem2.13} and $\beta_*\mathcal O_Z\simeq 
\mathcal O_{X_k}$ holds (see Lemma \ref{a-lem2.15}).  
Hence $(Z, \Delta_Z)$ is an analytic globally embedded simple 
normal crossing pair with all the desired properties. 
\end{step}
We finish the proof. 
\end{proof}

Let us prove Theorem \ref{a-thm1.1} (i). 

\begin{proof}[Proof of Theorem \ref{a-thm1.1} (i)] 
We take an arbitrary point $y\in Y$. 
We can freely shrink $Y$ around $y$ suitably. 
Hence we may assume that $Y$ is Stein and there exist 
$L$ and $K_X$ such that 
$\mathcal O_X(L)\simeq \mathcal L$ and 
$\mathcal O_X(K_X)\simeq \omega_X$, respectively. 
By perturbing the coefficients of $\Delta$, 
we may assume that $\Delta$ is a boundary $\mathbb Q$-divisor 
such that $N:=L-(K_X+\Delta)$ is an $f$-semiample $\mathbb Q$-divisor. 
By Lemma \ref{a-lem5.1}, we may further assume that $(X, \Delta)$ 
is an analytic globally embedded simple normal crossing pair and 
that the ambient space $M$ of $(X, \Delta)$ is projective 
over a polydisc $\Delta^m$ with $Y\hookrightarrow \Delta^m$. 
We put $\Sigma:=\lfloor \Delta\rfloor$. 
By taking a suitable Kawamata covering of the ambient space 
$M$ (see Lemma \ref{a-lem2.16}) 
and restricting it to $X$, we have a finite cover $p\colon V\to X$ such that 
$(V, \Sigma_V:=p^*\Sigma)$ is an analytic globally embedded simple 
normal crossing pair, $\Sigma_V$ is reduced, $N_V:=p^*N$ is 
Cartier, and $\mathcal O_X(K_X+\Sigma+\lceil N\rceil)=\mathcal O_X(L)$ 
is a direct summand of $p_*\mathcal O_V(K_V+\Sigma_V+N_V)$. 
We note that $(K_M+X)|_X=K_X$ and $(X, \Delta)$ is 
an analytic globally embedded simple normal crossing pair. 
As before, we can take a finite cyclic cover 
$q\colon W\to V$ such that 
$(W, \Sigma_W:=q^*\Sigma_V)$ is an analytic simple normal crossing pair, 
$\Sigma_W$ is reduced, and $\mathcal O_V(K_V+\Sigma_V+N_V)$ 
is a direct summand of $q_*\mathcal O_W(K_W+\Sigma_W)$ 
(for the details, see Step \ref{a-thm3.1-step3} in the proof of 
Theorem \ref{a-thm3.1} and Proof of 
Theorem \ref{a-thm1.1} (i) when $X$ is irreducible). 
By Theorem \ref{a-thm3.1} (i), 
every associated subvariety 
of $R^q(f\circ p\circ q)_*\mathcal O_W(K_W+\Sigma_W)$ 
is the $(f\circ p\circ q)$-image of some stratum of $(W, \Sigma_W)$. 
Therefore, every associated subvariety of $R^qf_*\mathcal O_X(L)$ is 
the $f$-image of some stratum of $(X, \Delta)$.  
\end{proof}

Finally, we prove Theorem \ref{a-thm1.1} (ii). 

\begin{proof}[Proof of Theorem \ref{a-thm1.1} (ii)] 
We take an arbitrary point $z\in Z$ and can freely 
shrink $Z$ around $z$. 
So we may assume that $Z$ is Stein. 
As usual, by perturbing 
the coefficients, 
we may assume that $L-(K_X+\Delta)\sim _{\mathbb Q}f^*H$, 
where $\mathcal O_X(L)\simeq 
\mathcal L$, $\mathcal O_X(K_X)\simeq \omega_X$, 
$H$ is a $\pi$-ample 
$\mathbb Q$-divisor on $Y$, 
such that $\Delta$ is a boundary $\mathbb Q$-divisor. 
We put $\Sigma:=\lfloor \Delta\rfloor$. 
By Lemma \ref{a-lem5.1}, we may further assume that 
$(X, \Delta)$ is an analytic globally embedded simple normal crossing 
pair and the ambient space $M$ of $(X, \Delta)$ is 
projective over a polydisc $\Delta^m$ with $Z\hookrightarrow \Delta^m$. 
By taking a suitable Kawamata covering of the ambient space 
$M$ (see Lemma \ref{a-lem2.16}) and restricting it to $X$, 
we get a finite cover $p\colon V\to X$ such that 
$(V, \Sigma_V:=p^*\Sigma)$ is an analytic globally 
embedded simple normal crossing pair, $\Sigma_V$ is reduced, 
$N_V:=p^*N$ is Cartier, 
and $\mathcal O_X(L)$ is a direct summand of 
$p_*\mathcal O_V(K_V+\Sigma_V+N_V)$. 
As in Step \ref{a-thm3.1-step3} in 
the proof of Theorem \ref{a-thm3.1}, 
by taking a tower of cyclic covers, we can construct a finite 
cover $q\colon W\to V$ such that 
$(W, \Sigma_W:=q^*\Sigma_V)$ is an analytic simple normal crossing pair, 
$\Sigma_W$ is reduced, $N_W:=q^*N_V$ is free over $Z$, 
and $\mathcal O_V(K_V+\Sigma_V+N_V)$ is a 
direct summand of $q_*\mathcal O_W(K_W+\Sigma_W+N_W)$. 
Hence it is sufficient to prove 
\begin{equation*}
R^p\pi_*R^q(f\circ p\circ q)_*\mathcal O_W(K_W+\Sigma_W+N_W)
=0
\end{equation*} 
for every $p>0$. Since Theorem \ref{a-thm3.1} (ii) holds for 
analytic simple 
normal crossing pairs, the argument in the proof of 
Theorem \ref{a-thm1.1} (ii) when $X$ is irreducible 
in Section \ref{a-sec4} 
works without any changes. 
Thus we get the desired vanishing theorem. 
\end{proof}

\section{Proof of Theorem \ref{a-thm1.2}}\label{a-sec6} 
In this section, we will prove Theorem \ref{a-thm1.2}. 

\begin{proof}[Proof of Theorem \ref{a-thm1.2}]
We take an arbitrary point $z\in Z$. 
It is sufficient to prove that $R^p\pi_*R^qf_*\mathcal L=0$ holds 
for every $p>0$ in a neighborhood of $z\in Z$. 
By Lemma \ref{a-lem5.1}, 
we may assume that 
the following commutative diagram 
\begin{equation*}
\xymatrix{
X \ar[d]_-{\pi\circ f}\ar@{^{(}->}[r]^-\iota& M\ar[d] \\ 
Z \ar@{^{(}->}[r]_-{\iota_Z}& \Delta^m
}
\end{equation*}
exists, where $(X, \Delta)$ is an analytic globally embedded simple 
normal crossing pair, $M$ is the ambient space of $(X, \Delta)$ and 
is projective over $\Delta^m$, and $Z$ is 
a closed analytic subspace of $\Delta^m$ with 
$\iota_Z(z)=0\in \Delta^m$. 
Since $Z$ is Stein and $f$ and $\pi$ are both 
projective, there exist Cartier divisors $L$ and $K_X$ such that 
$\mathcal L\simeq \mathcal O_X(L)$ and $\omega_X\simeq 
\mathcal O_X(K_X)$, respectively, and 
that $L-(K_X+\Delta)\sim _{\mathbb R} f^*H$ holds for some 
$\mathbb R$-Cartier divisor $H$, which is nef 
and log big over $Z$ with respect to $(X, \Delta)$, on $Y$. 
\setcounter{step}{0}
\begin{step}\label{a-thm1.2-step1}
In this step, we will prove the desired vanishing theorem 
under the extra assumption that 
every stratum of $(X, \Delta)$ is dominant onto some 
irreducible component of $f(X)$. 

From now on, we assume that every stratum of $(X, \Delta)$ is 
dominant onto some irreducible component of $f(X)$. 
By taking the Stein factorization, we may assume that $f_*\mathcal O_X\simeq 
\mathcal O_Y$ holds. 
In particular, $Y$ is reduced. 
Let $X^\dag$ be any connected component of $X$. 
Since $f$ has connected fibers, 
$Y^\dag:=f(X^\dag)$ is an irreducible component of $Y$ and every irreducible 
component of $X^\dag$ is mapped to $Y^\dag$ by $f$. Hence, 
we may further assume that 
$X$ is connected and that $Y$ is irreducible. 
Since $H$ is $\pi$-big, 
we can write $H=E+A$, where $A$ is a $\pi$-ample 
$\mathbb R$-divisor and $E$ is an effective $\mathbb R$-Cartier 
divisor. We take a projective bimeromorphic morphism 
$\alpha\colon M'\to M$ from a smooth complex variety $M'$, which 
is an isomorphism outside $\Supp f^*E$. 
Let $X'$ be the strict transform of $X$ on $M'$. 
We put $\varphi:=\alpha|_{X'}\colon X'\to X$. 
By taking $\alpha\colon M'\to M$ suitably, we 
may assume that $(X', \Sigma)$ is an analytic globally embedded simple 
normal crossing pair, $M'$ is the ambient space of $(X', \Sigma)$, and 
$\Sigma$ contains $\Supp \varphi^*\Delta$ and 
$\Supp \varphi^*f^*E$ (see \cite[Theorems 13.3 and 12.4]{bierstone-milman}). 
For $k\gg 1$, we can write 
\begin{equation*}
K_{X'}+\Delta'=\varphi^*\left(K_X+\Delta+\frac{1}{k}f^*E\right)+E'
\end{equation*} 
such that 
\begin{itemize}
\item[(1)] $(X', \Delta')$ is an analytic 
globally embedded simple normal crossing 
pair such that $\Delta'$ is a boundary $\mathbb R$-divisor, 
\item[(2)] $M'$ is the ambient space of $(X', \Delta')$, 
\item[(3)] $E'$ is an effective $\varphi$-exceptional 
Cartier divisor on $X'$, and 
\item[(4)] $\varphi$ maps every stratum of $(X', \Delta')$ 
bimeromorphically onto some stratum of $(X, \Delta)$. 
\end{itemize}
We put $L':=\varphi^*L+E'$. 
Then we have $\varphi_*\mathcal O_{X'}(L')\simeq 
\mathcal O_X(L)$ and $R^i\varphi_*\mathcal O_{X'}(L')=0$ for 
every $i>0$ by Lemma \ref{a-lem2.13} since 
$L'-(K_X'+\Delta')=\varphi^*(L-(K_X+\Delta+\frac{1}{k}f^*E))$. We note that 
\begin{equation*}
\begin{split}
L'-(K_{X'}+\Delta')& = \varphi^*L+E'-\varphi^*\left( 
K_{X}+\Delta+\frac{1}{k}f^*E\right)-E' \\ 
&= \varphi^*(L-(K_X+\Delta))-\frac{1}{k}\varphi^*f^*E\\
&\sim_{\mathbb R} \varphi^*\left(f^*H-\frac{1}{k}f^*E\right)\\ 
&=(f\circ \varphi)^*\left(\frac{1}{k}A+\frac{k-1}{k}H\right) 
\end{split} 
\end{equation*} 
and that $\frac{1}{k}A+\frac{k-1}{k}H$ is $\pi$-ample. 
Thus, by Theorem \ref{a-thm1.1} (ii), 
we obtain 
\begin{equation*}
R^p\pi_*R^qf_*\mathcal O_X(L)=
R^p\pi_*R^q(f\circ \varphi)_*\mathcal O_{X'}(L')=0
\end{equation*} 
for every $p>0$. 
\end{step}
\begin{step}\label{a-thm1.2-step2} 
In this step, we will 
treat the general case by induction 
on $\dim f(X)$. 
We note that the desired vanishing theorem holds 
when $\dim f(X)=0$ by Step 
\ref{a-thm1.2-step1}. 
By using Lemma \ref{a-lem2.17} finitely many times, 
we can decompose $X=X'\cup X''$ as follows:~$X'$ is the union of 
all strata of $(X, \Delta)$ that are not mapped to irreducible components 
of $f(X)$ and $X''=X-X'$. We put 
\begin{equation*}
K_{X''}+\Delta_{X''}:=(K_X+\Delta)|_{X''}-X'|_{X''}. 
\end{equation*} 
Then $f\colon (X'', \Delta_{X''})\to X$ and $L'':=L|_{X''}-X'|_{X''}$ 
satisfy the assumption in Step \ref{a-thm1.2-step1}. 
We consider the following short exact sequence: 
\begin{equation*}
0\to \mathcal O_{X''}(L'')\to \mathcal O_X(L)\to \mathcal O_{X'}(L)\to 0. 
\end{equation*}
We note that every associated subvariety 
of $R^qf_*\mathcal O_{X''}(L'')$ is an irreducible component of $f(X)$ 
by Theorem \ref{a-thm1.1} (i) 
and that every associated subvariety 
of $R^qf_*\mathcal O_{X'}(L)$ is contained in $f(X')$ for every $q$. 
Therefore, the connecting homomorphisms 
\begin{equation*}
\delta\colon 
R^qf_*\mathcal O_{X'}(L)\to R^{q+1}f_*\mathcal O_{X''}(L'')
\end{equation*} 
are zero for all $q$. Thus we have the following short 
exact sequence 
\begin{equation*}
0\to R^qf_*\mathcal O_{X''}(L'')\to R^qf_*\mathcal O_X(L)\to R^qf_*\mathcal 
O_{X'}(L)\to 0
\end{equation*} 
for every $q$. By Step \ref{a-thm1.2-step1}, 
we have $R^p\pi_*R^qf_*\mathcal O_{X''}(L'')=0$ for 
every $p>0$. 
On the other hand, $R^p\pi_*R^qf_*\mathcal O_{X'}(L)=0$ 
for every $p>0$ by induction on $\dim f(X)$. 
This implies that 
\begin{equation*}
R^p\pi_*R^qf_*\mathcal O_X(L)=0
\end{equation*} 
holds for every $p>0$. 
\end{step}
We finish the proof. 
\end{proof}

\end{document}